\documentclass[noinfoline]{imsart}

\RequirePackage[OT1]{fontenc}
\RequirePackage{amsthm,amsmath,amssymb,bbm}
\RequirePackage{verbatim}
\RequirePackage[numbers]{natbib}
\RequirePackage[colorlinks,citecolor=blue,urlcolor=blue]{hyperref}
\usepackage{enumerate}
\usepackage{tikz}
\usetikzlibrary{arrows}

\usepackage{tikz-cd}
\usetikzlibrary{decorations.markings}
\tikzset{degil/.style={
            decoration={markings,
            mark= at position 0.5 with {
                  \node[transform shape] (tempnode) {$\slash$};
                  }
              },
              postaction={decorate}
}
}

\arxiv{}

\startlocaldefs
\newtheorem{theorem}{Theorem}
\newtheorem{definition}{Definition}
\newtheorem{proposition}{Proposition}

\numberwithin{equation}{section}
\theoremstyle{plain}

\endlocaldefs

\newcommand{\R}{\mathbb{R}}
\newcommand{\N}{\mathbb{N}}

\renewcommand{\P}{\mathbb{P}}

\newtheorem{remark}{Remark}

\begin{document}

\begin{frontmatter}
\title{Pathwise vs.~Path-by-Path Uniqueness}
\runtitle{Pathwise vs.~Path-by-Path Uniqueness}

\begin{aug}
\author{\fnms{Alexander} \snm{Shaposhnikov}\ead[label=e1]{
shal1t7@mail.ru}}
\and
\author{\fnms{Lukas} \snm{Wresch}\ead[label=e2]{wresch@math.uni-bielefeld.de}}

\runauthor{A. Shaposhnikov, L. Wresch}

\affiliation{Bielefeld University and Technion, Israel Institute of Technology}

\address{Alexander Shaposhnikov\\
Faculty of Mathematics, Bielefeld University\\
Bielefeld, Germany\\
\printead{e1}}

\address{Lukas Wresch\\
Faculty of Industrial Engineering\\
Technion, Israel Institute of Technology\\
Haifa, Israel\\
\printead{e2}}

\end{aug}

\begin{abstract}
We construct a series of stochastic differential equations
of the form $dX_t = b(t, X_t)  dt +  dB_t$ which exhibit \textit{nonuniqueness}
in the \textit{path-by-path} sense while having a unique adapted solution
in the sense of stochastic processes, i.e. \textit{pathwise} uniqueness holds. 
\end{abstract}


\end{frontmatter}

\section{Introduction}
In this paper we consider the stochastic differential equation
\begin{equation}\label{eq:abstract_sde}
dX_t = b(t, X_t)\,dt + dB_t, \ X_0 = x_0,
\end{equation}
where $b \colon [0,T]\times \mathbb{R}^d \rightarrow  \mathbb{R}^d$ is a Borel measurable mapping, $(B_t)_{t \geq 0}$ is a standard 
$d$--dimensional Brownian motion, $x_0 \in \mathbb{R}^d$.
Let us recall that a solution (\textit{weak} solution) to SDE~\eqref{eq:abstract_sde} is
a pair of a Brownian motion $(B_t)_{t \geq 0}$ 
defined on some filtered probability space
$(\Omega, \mathcal F, (\mathcal F_t)_{t \geq 0})$ and
a stochastic process $(X_t)_{t \in [0, T]}$ adapted to the filtration
$(\mathcal F_t)_{t\in[0,\infty)}$ such that $\P$-a.s.
\begin{equation}\label{eq:abstract_ode}
X_t = x_0 + \int_{[0, t]} b(s, X_s) \, ds  + B_ t , \ t \in [0, T].
\end{equation}
The solution $(B, X)$ is called a \textit{strong} solution if the process
$(X_t)_{t \in [0, T]}$ is adapted to the augmented filtration (i.e.~the completed filtration)
 $(\overline{F^{B}})_{t \geq 0}$
generated by the Brownian motion.

\begin{definition}\label{de:weak_uniqueness}
For SDE \eqref{eq:abstract_sde} weak uniqueness holds if for any two solutions $(B, X)$, $(\widetilde{B}, Y)$  (which may be defined on different filtered probability spaces) one has
$$
\mathrm{Law}(X) = \mathrm{Law}(Y).
$$
\end{definition}

\begin{definition}\label{de:pathwise_uniqueness}
For SDE \eqref{eq:abstract_sde} pathwise uniqueness holds if for any two solutions $(B, X)$, $(B, Y)$ defined on the same filtered probability space with the same Brownian motion $(B_t)_{t \geq 0}$
there exists a measurable set $\Omega' \subseteq \Omega$ with $\P[\Omega'] = 1$ such that
$$
X_t(\omega) = Y_t(\omega), \qquad \omega \in \Omega', t \in[0,T].
$$
\end{definition}

At the same time one can consider random ordinary differential equation~\eqref{eq:abstract_ode} and ask whether the uniqueness holds in the
pure ODE setting.
\begin{definition}\label{de:path_by_path_uniqueness}
For  SDE \eqref{eq:abstract_sde} \textit{path-by-path} uniqueness holds
if  there exists a measurable set 
$\Omega' \subset C([0,T], \mathbb{R}^d)$ of full Wiener measure such
that for any Brownian trajectory from $\Omega'$
 integral equation \eqref{eq:abstract_ode} has a unique solution.
\end{definition}

Let us point out that in Definition \ref{de:pathwise_uniqueness}
of  \textit{pathwise uniqueness}
the set $\Omega' \subset \Omega$ of full measure a priori is allowed to depend  on
the both processes $X$ and $Y$.
This is in stark contrast to \textit{path-by-path} uniqueness, where there is a set of full measure $\Omega'' := B^{-1}(\Omega')$, where
$\Omega' \subset C([0, T], \mathbb{R}^d)$ is the set of ``good''
Brownian trajectories from Definition \ref{de:path_by_path_uniqueness}, 
such that the functions $t \rightarrow X_{t}(\omega)$ and $t \rightarrow Y_t(\omega)$ have to coincide for all $\omega \in \Omega''$.
Furthermore, we call a map 
$
\Omega \supseteq \Omega' \rightarrow X(\omega)\in C([0,T], \mathbb{R}^d)$
a \textit{path-by-path solution} if $\mathbb{P}[\Omega']=1$ and $X(\omega)$ solves  ODE \eqref{eq:abstract_ode} for all $\omega\in\Omega'$. Note that \textit{path-by-path} solutions are not required to be adapted to the filtration 
$(\mathcal F_t)_{t \geq 0}$  but every solution of the corresponding SDE yields a \textit{path-by-path} solution.
The diagram below represents the connections between different concepts of \textit{existence} and \textit{uniqueness}
of solutions to SDE \eqref{eq:abstract_sde} which immediately follow from the definitions.

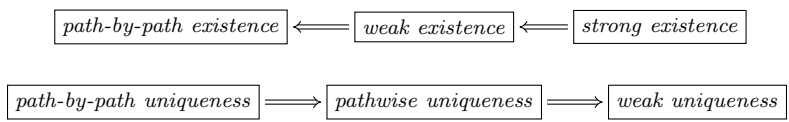
\begin{figure}[h]
\begin{tikzpicture} [node distance=3.0cm]
\node(a)[rectangle, draw=black]
at (0, 0)
{
\textit{path-by-path existence}
};

\node(b)[rectangle, draw=black]
at (3.5, 0)
{
\textit{weak existence}
};

\node(c)[rectangle, draw=black] 
at (6.5, 0)
{
\textit{strong existence}
};

\draw[implies-,
shorten >= 2pt, shorten <= 2pt, 
double equal sign distance] 
(a) -- (b);

\draw[implies-,
shorten >= 2pt, shorten <= 2pt, 
double equal sign distance] 
(b) -- (c);

\node(d)[rectangle, draw=black]
at (-0.5, -1)
{
\textit{path-by-path uniqueness}
};

\node(e)[rectangle, draw=black]
at (3.5, -1)
{
\textit{pathwise uniqueness}
};

\node(f)[rectangle, draw=black]
at (7, -1)
{
\textit{weak uniqueness}
};

\draw[-implies,
shorten >= 2pt, shorten <= 2pt, 
double equal sign distance] 
(d) -- (e);

\draw[-implies,
shorten >= 2pt, shorten <= 2pt, 
double equal sign distance] 
(e) -- (f);

\end{tikzpicture}
\end{figure}

The classical example of a SDE which has a \textit{weak} solution
but no \textit{strong} solutions is Tanaka's equation
\begin{equation}\label{eq:tanaka}
dX_t = \operatorname{sgn}(X_t)\,dB_t, \ X_0 = 0.
\end{equation}
For SDE \ref{eq:tanaka}
\textit{weak uniqueness} holds but 
\textit{pathwise uniqueness} does not (see e.g. \cite{KS98}, p.~301,  Example~3.5).
It is also worth mentioning the celebrated example due to B. Tsirelson 
(see~\cite{Ts75}) of a SDE of the form
$$
 dX_t = b(X_{\leq t}, t)\, dt +  dB_t, \ X_0  = 0,
$$
where $b$ is a bounded Borel measurable function of $t$ and
the ``past'' of $X$ up to the time $t$, 
which admits a unique (in the \textit{weak} sense) \textit{weak} solution but no \textit{strong} solutions and \textit{pathwise uniqueness} does not hold.

The question whether every \textit{path-by-path} solution can be obtained from a \textit{weak} solution to the SDE was posed as an open problem in \cite{AL17} and also was mentioned in the book \cite{Fla15} (see the discussion on p.~12).
In this paper we show that, in general, this is not true. Moreover, we construct SDEs  such that a \textit{strong} solution exists, \textit{pathwise} uniqueness holds, but \textit{path-by-path} uniqueness fails to hold.

\medskip

Concerning the historical development of \textit{path-by-path} uniqueness to our knowledge the first result was obtained by A.~M.~Davie in \cite{Dav07} for the case when $b$ is Borel measurable and bounded. Later, Davie extended his result and proved that \textit{path-by-path} uniqueness holds in the non-degenerate multiplicative noise case (see \cite{Dav11}).
The original result of Davie was established with a different method by
the first author in \cite{Sh16} 
(see also some corrections in \cite{Sh17}), which
enabled him to present a simpler proof of the main theorem
from \cite{Dav07} and strengthen it in multiple directions.
In particular, in \cite{Sh16} \textit{path-by-path} uniqueness was obtained for some unbounded drift coefficients $b$ and 
by carefully examining the arguments one can show that  the set $\Omega' \subseteq \Omega$ where 
\textit{path-by-path} uniqueness holds can be constructed independently of the initial condition.
R.~Catellier and M.~Gubinelli in \cite{CG16} showed that 
\textit{path-by-path} uniqueness can be established if the Wiener process is replaced by a fractional Brownian motion in $\R^d$ with Hurst parameter $H$. Furthermore, the drift $b$ was allowed to be merely a distribution as long as $H$ was sufficiently small.
In the work \cite{BFGM14} L.~Beck, F.~Flandoli, M.~Gubinelli and M.~Maurelli proved that \textit{path-by-path} uniqueness does not only hold for SDEs, but also for SPDEs.
In \cite{Pri18} E.~Priola considered equations driven by a L\'evy process such that the L\'evy measure fulfills some integrability condition, see also \cite{Zh18} for related results.
In 2016 O.~Butkovsky and L.~Mytnik showed in \cite{BM16} that 
\textit{path-by-path} uniqueness holds for the stochastic heat equation with space-time white noise for bounded Borel measurable drifts.
In the works \cite{Wre16, Wre17} the second author established \textit{path-by-path} uniqueness for the case where $\R^d$ is replaced by a Hilbert space $H$ and $B$ is a cylindrical Wiener process as long as the linear negative operator is added to the SDE and the nonlinear part is bounded with respect to a specific norm, a condition which is trivial if $\operatorname{dim} H < \infty$.
In the recent paper \cite{Pri19} E.~Priola
improved the aforementioned result 
by allowing a time--dependent coefficient in front of the L\'evy noise which can in essence be as degenerate as in the condition for  \textit{pathwise} uniqueness.

\medskip

In conclusion, \textit{path-by-path} uniqueness can be established when the drift is singular and also in the case when the noise term is degenerate.
However, in general, the conditions to establish \textit{path-by-path} uniqueness are stricter than those for \textit{pathwise} uniqueness so in some cases there is a ``gap'' between the available \textit{pathwise} and the \textit{path-by-path} results. In this paper we would like to add a new point of focus. By carefully constructing a SDE we can determine that any global ``solution'' must know something about its own future and cannot be adapted to a filtration with respect to which the ``driving'' stochastic process remains a Brownian motion. Next, we can construct a SDE having a unique adapted solution and with probability one having some other non--adapted solutions to the corresponding ODE.

\section{Bessel processes}
The constructions in the next sections are based on the properties of the stochastic differential
equations governing  Bessel processes. For the sake of completeness below we recall 
the known results which will be used in the subsequent considerations.

\begin{definition}
For $\delta > 0$, $Z_0 \geq 0$ the unique strong solution of the SDE
\begin{equation}\label{eq:square_bessel_sde}
Z_{t} = Z_0 + \delta t + 2\int_{[0, t]}\sqrt{|Z_t|}\, dB_s
\end{equation}
is called the square of the $\delta$--dimensional Bessel process started at $Z_0$.
\end{definition}
We refer to Chapter 11 in  \cite{RY99} for the basic properties of this equation.
In particular, it is well--known that although the diffusion coefficient is non--Lipschitz for equation \eqref{eq:square_bessel_sde}
\textit{pathwise} uniqueness holds and, moreover, with probability one the solution is non-negative.
The process $\sqrt{Z_t}$ is called the $\delta$--dimensional Bessel process started at 
$\sqrt{Z_0}$.

Now let us introduce for $\delta > 1$ the Bessel SDE
\begin{equation}\label{eq:bessel_sde}
X_{t} = X_{0} + \int_{[0, t]}\mathbbm{1}_{X_{s} \neq 0}\frac{\delta - 1}{2X_s}\,ds +  B_t.
\end{equation}
One can verify that for $\delta > 1 $ the process  $\sqrt{Z_t}$ satisfies \eqref{eq:bessel_sde} with
$X_0 = \sqrt{Z_0}$.

The next theorem was obtained by A. Cherny in \cite{Cherny2000}.
\begin{theorem}\label{th:cherny_uniqueness_theorem}
 For  SDE \eqref{eq:bessel_sde}
\begin{enumerate}
\item if $\delta > 1$, $X_0 \geq 0$ then the $\delta$--dimensional Bessel process
is the unique non--negative solution,
moreover, it is a strong solution,
\item if $\delta \geq 2$, $X_0 \neq 0$ then pathwise uniqueness holds,
\item if $1 < \delta < 2 $ or $X_0 = 0$ then there exist other strong solutions 
with the same $X_0$ and $B$, there exist weak solutions which are not strong, the uniqueness in law does not
hold.
\end{enumerate}
\end{theorem}

\begin{remark} For $\delta \geq 2$ if $(X, B)$ is a weak solution
to SDE \ref{eq:bessel_sde} then $X_t$ does not change its sign on 
$(0, \infty)$. This follows by the classical comparison theorem (see e.g Theorem 3.7 in \cite{RY99} or Proposition 2.18 in \cite{KS98}) applied to SDE \ref{eq:square_bessel_sde} and taking into account that for $\delta = 2$ one has the identity
$\mathrm Law(Z) = \mathrm Law(|\widetilde{B}|)$, where 
$\widetilde{B}$ is a standard $2$--dimensional Brownian motion.
\end{remark}

Now let us recall the SDE for the $3$--Bessel bridge with the terminal value~$1$:
\begin{equation}\label{eq:3_bessel_bridge}
X_{t} = X_0 + \int_{[0, t]} \Bigl(\frac{1 - X_s}{1 - s} + \frac{1}{X_s}\Bigr)\,ds + B_{t}, 
\ X_0 \geq 0, \ t \in [0, 1]
\end{equation}
see e.g.  equation (29) on p.~274 in \cite{Pit99}. 
In fact, we will be interested in the SDE
\begin{equation}\label{eq:main_helper_sde}
 dX_t = f(t, X_t)\,dt + dB_t, \ X_0 = x_0 \in \R,\ t \in [0, 1],
\end{equation}
where 
$$
f(t, x) := \mathbbm{1}_{\{x > 0\}}
\Bigl(
\frac{1 - x}{1 - t} + \frac{1}{x}
\Bigr)
- 
\mathbbm{1}_{\{x < 0\}}
\Bigl(
\frac{1 - (-x)}{1 - t} + \frac{1}{-x}
\Bigr).
$$
As we shall see in the next proposition for $x_0 = 0$ SDE \ref{eq:main_helper_sde} allows to select the 
$3$--Bessel bridge or the $(-1)\times 3$--Bessel bridge as its solution.
\begin{proposition}\label{pr:helper_equation_properties}
For SDE \eqref{eq:main_helper_sde}
\begin{enumerate} 
\item For $x_0 = 0$ there is a unique nonnegative weak solution 
(unique in the \textit{pathwise} sense among all nonnegative solutions),
and  analogously there is a unique nonpositive weak solution.
Moreover, these solutions are \textit{strong} solutions.
\item For $x_0 = 0$ any weak solution with probability one preserves its sign for $t \in (0, 1]$, in particular, never reaches $0$ for $t \in (0, 1]$ and equals $1$~or~$-1$ when $t = 1$.
\end{enumerate}
\end{proposition}
\begin{proof}
Let us notice that there is a nonnegative \textit{weak} solution to
 SDE \eqref{eq:main_helper_sde} given by the $3$-Bessel bridge with the terminal value $1$  and as it is well-known this solution with probability $1$ never reaches $0$ for $t \in (0, 1]$. The corresponding nonpositive solution is given by the $(-1)\times 3$--Bessel bridge.
Now let us establish \textit{pathwise} uniqueness in the class of nonnegative solutions, the case of nonpositive solutions is handled completely analogously. One can see that it is sufficient to establish \textit{pathwise} uniqueness on every interval $[0, T]$, $T \in (0, 1)$.
Let us assume that $(X_1, B)$, $(X_2, B)$ are two \textit{weak} solutions on $[0, T]$ to SDE \eqref{eq:main_helper_sde}  
defined on some filtered probability space 
$(\Omega, \mathcal F, (\mathcal F_t)_{t \geq 0})$
with the same Brownian motion $B$ and $X_1, X_2 \geq 0$ a.s. 
Let us set
$$
\varrho_{T} := \exp\Biggl(
-\int_{[0, T]}\frac{1}{1 - t}\, dt - 
\frac{1}{2}
\int_{[0, T]}\frac{1}{(1 - t)^2}\,dB_t
\Biggr).
$$
Then a.s. $\varrho_T > 0$ and $\mathbb{E}\varrho_T = 1$.
By Girsanov's theorem under the new measure 
$$
 dQ := \varrho_T\, dP
$$
the process 
$$
\widetilde{B}_t := \int_{[0, t]}\frac{1}{1 - s}\, ds + B_t, \ t \in [0, T].
$$
is a Brownian motion with respect to the same filtration 
$(\mathcal F_t)_{t \geq 0}$. It is also clear that the filtrations
$(\overline{\mathcal{F}^{B}}_t)_{t \in [0, T]}$ and
$(\overline{\mathcal{F}^{\widetilde{B}}}_t)_{t \in [0, T]}$ coincide.
Now one can see that on the probability space
$(\Omega, \mathcal F, (\mathcal F_t)_{t \in [0, T]}, Q)$
the processes $(X_1, \widetilde{B})$, $(X_2, \widetilde{B})$ are \textit{weak}
nonnegative solutions to the SDE
\begin{equation}\label{eq:3_bessel_bridge}
X_{t} = \int_{[0, t]} \mathbbm{1}_{\{X_t > 0\}} \Bigl(\frac{- X_t}{1 - t} + \frac{1}{X_t}\Bigr)\, dt + \widetilde{B}_{t},
\ X_0 = 0, \ t \in [0, T].
\end{equation}
Let us define $Y_{1, t} := X^{2}_{1, t}$, 
$Y_{2, t} := X^{2}_{2, t}$.  Applying Ito's formula one can show that
$(Y_{1, t}, \widetilde{B})$ and $(Y_{2, t}, \widetilde{B})$ are
\textit{weak}
nonnegative solutions to the SDE
\begin{equation}\label{eq:square_bessel_bridge_sde}
Y_{t} = \int_{[0, t]}\frac{-2Y_t}{1 - t}\, dt + 2t + \int_{[0, t]}\sqrt{|Y_t|}\,d\widetilde{B}_t, \ t \in [0, T],
\end{equation}
where we have also used the occupation time formula for semimartingales (see e.g.~\cite{RY99}, Ch. 7) applied to $X = X_1, X_2$:
$$
\int_{[0, t]}\mathbbm{1}_{\{X_s = 0\}}\, ds =
\int_{\mathbb{R}}\mathbbm{1}_{\{x = 0\}}L^{x}_{t}(X)\, dx = 0, \ t \in [0, T].
$$
For SDE \eqref{eq:square_bessel_bridge_sde} the classic Yamada--Watanabe condition is statisfied and \textit{pathwise} uniqueness follows, see e.g.
\cite{RY99}, Ch. 9, Theorem 3.5. Consequently, by the Yamada--Watanabe theorem there is a unique \textit{weak} solution to \eqref{eq:square_bessel_bridge_sde} and the solution is \textit{strong}. 
Since $X_{1, t} = \sqrt{|Y_{1, t}|}$ and $X_{2, t} = \sqrt{|Y_{2, t}|}$ this easily gives the required a.s. equality 
$$
X_{1, t} = X_{2, t},\ t \in [0, T],
$$
and applying the Yamada--Watanabe theorem we obtain that
the solution $(X_1, \widetilde{B})$ is \textit{strong}.
Taking into account the equality
$$
(\overline{\mathcal{F}^{B}}_t)_{t \in [0, T]} =
(\overline{\mathcal{F}^{\widetilde{B}}}_t)_{t \in [0, T]}
$$
we prove the first claim of
Proposition~\ref{pr:helper_equation_properties}.

Let $(X, B)$ be a \textit{weak} 
solution to SDE \eqref{eq:main_helper_sde} with $x_0 = 0$.
Let us define the nonincreasing sequence of Markov moments $\{\tau_n\}$ 
$$
\tau_n := \inf \bigl\{t > 0: |X_{\tau_n}| = 1/n\bigr\}.
$$
Since $f(t, 0) \equiv 0$ for any $t' > 0$ on the set 
$$
\Omega_{t'} := \{\lim_{n \to \infty}\tau_n > t'\}
$$
we a.s. have the equalities
$$
X_{t} = 0, \ X_t = B_t, \ t \in [0, t'],
$$
that is possible only if $P(\Omega_{t'}) = 0$. Consequently,
$\lim_{n \to \infty}\tau_n = 0$ a.s.
It is easy to notice that the arguments presented above in
the case $x_0 = 0$ similarly yield \textit{pathwise} uniqueness
for the SDE
$$
Z_{t} = Z_{\tau} + \int_{[\tau, t]}f(s, Z_{s})\, ds +  B_t - B_{\tau}, \ t \in [\tau, T],
$$
as soon as $Z_{\tau} \neq 0$ with probability $1$ on the set 
$\{\tau < T\}$.
Applying this observation to $\tau := \tau_n\wedge T$ for each $n$
and taking into account that the $3$--Bessel bridge does not change its sign we obtain the desired claim.
\end{proof}

We will also need the following classic example of a singular SDE which has no solutions at all.

\begin{proposition}
The SDE
\begin{equation}\label{eq:no_solutions}
X_{t} = \int_{[0, t]}\mathbbm{1}_{\{X_{s}\neq 0\}}\frac{-1}{2X_s}\,
ds + B_t, \ X_0 = 0.
\end{equation}
has no \textit{weak} solutions.
\end{proposition}
\begin{proof}
For the proof see Example 2.1 in \cite{Cherny2001}.
\end{proof}

\section{Equation without adapted solutions}

\begin{theorem}\label{th:no_adapted_solutions}
There exists a Borel mapping $b = (b_1, b_2): [0, 2] \times \mathbb{R}^2 \to \mathbb{R}^2$,
\begin{equation}\label{eq:no_adapted_solutions}
X_{t} = \int_{[0, t]}b(s, X_s)\, ds + B_t,\ X_0 = 0, \ t \in [0, 2]
\end{equation}
such that
\begin{enumerate}
\item there exists a set of Brownian trajectories $\Omega$ 
of full measure, such that for every $B \in \Omega$  integral equation \eqref{eq:no_adapted_solutions} has at least one solution
(understood in the pure ODE sense) defined on the whole interval $[0, 2]$,
\item equation \eqref{eq:no_adapted_solutions} considered as a SDE
has no weak solutions $(X, B)$ defined on the whole  interval $[0, 2]$. 
\end{enumerate}
\end{theorem}
\begin{proof}
Before proceeding to the formal proof we would like to briefly outline the strategy.
On the interval $[0, 1]$ a \textit{weak} solution to SDE~\ref{eq:no_adapted_solutions} exists and its first component is a Brownian motion while the second one is either a $3$--Bessel bridge or $(-1)\times 3$--Bessel bridge.
On $[1, 2]$ depending on the sign at $1$ of each component, either there is no \textit{weak} solution at all or there is one. Furthermore, the second component cannot change its sign. It implies that the sign of the second component on 
$[0, 1]$ depends on the sign of the first one at the end of the interval provided the solution exists on the whole interval $[0, 2]$, hence the resulting process is not adapted to the original filtration.

{\bf Step 1.}
For $t \in [0, 1]$ let us set
$$
b_1 := 0,
$$
$$
b_2 := \mathbbm{1}_{\{x_2 > 0\}}
\Bigl(
\frac{1 - x_2}{1 - t} + \frac{1}{x_2}
\Bigr)
- 
\mathbbm{1}_{\{x_2 < 0\}}
\Bigl(
\frac{1 - (-x_2)}{1 - t} + \frac{1}{-x_2}
\Bigr).
$$
In this case any solution $(X, B)$ to \eqref{eq:no_adapted_solutions}
for $t \in [0, 1]$ is of the following form:
\begin{equation}\label{eq:1_0_1}
X_{1,t} = B_{1,t},  \ t \in [0, 1]
\end{equation}
\begin{equation}\label{eq:2_0_1}
X_{2,t} = \int_{[0, t]}b_2(s, X_{2,s})\,ds + B_{2,t}, \
t \in [0, 1]
\end{equation}
Now let us remind that by 
Proposition~\ref{pr:helper_equation_properties} 
for SDE \eqref{eq:2_0_1}
\begin{enumerate}
\item there exists a unique nonnegative solution, i.e.~the $3$--Bessel bridge from $0$ to $1$,
\item there exists a unique  nonpositive solution, 
i.e.~the $(-1)\times 3$--Bessel bridge from $0$~to~$-1$,
\item any solution  does not change its sign.
\end{enumerate}

{\bf Step 2.}
Now we would like to make use of
 equation \eqref{eq:no_solutions}, which will play the role of a 
``randomized filter''.  The idea is to force the equation to ``select'' 
solutions in an non--adapted way.

For $t \in (1, 2]$ let us set
$$
b_1 := \mathbbm{1}_{\{x_1 \neq 0\}}\frac{1}{x_1}
$$
\begin{multline*}
b_2 := \mathbbm{1}_{\{x_1 > 0\}}\mathbbm{1}_{\{x_2 > 0\}}
\frac{-1}{2(x_2 - 1)} +
\mathbbm{1}_{\{x_1 < 0\}}\mathbbm{1}_{\{x_2 < 0\}}
\frac{-1}{2(x_2 + 1)} \\
+\mathbbm{1}_{\{x_1 > 0\}}\mathbbm{1}_{\{x_2 < 0\}}\frac{1}{x_2} + 
\mathbbm{1}_{\{x_1 < 0\}}\mathbbm{1}_{\{x_2 > 0\}}\frac{1}{x_2}.
\end{multline*}

Any solution $(X, B)$ to \eqref{eq:no_adapted_solutions}
for $t \in [1, 2]$ has the form
\begin{equation}\label{eq:1_2_3}
X_{1,t} =  X_{1, 1} + \int_{[1, t]}b_1(s, X_{1,s})\, ds + 
B_{1,t} - B_{1,1}
\end{equation}
\begin{equation}\label{eq:2_2_3}
X_{2,t} = X_{2, 1} + \int_{[1, t]}b_2(s, X_{1,s}, X_{2,s})\, ds + B_{2,t} - B_{2,1}.
\end{equation}
Let us remind that the drift 
$
 x \mapsto \mathbbm{1}_{\{x \neq 0\}}\frac{1}{x}
$
corresponds to the $3$--Bessel SDE, therefore
equation \eqref{eq:1_2_3} has a unique solution and this solution
preserves its sign on the interval $[1, 2]$.
In turn, equation \eqref{eq:2_2_3} has no solutions if 
$X_{1, 1} \cdot X_{2, 1} > 0$ and has a unique solution otherwise.
\vskip .1in
Let us consider a ``global'' solution 
$X$ to SDE~\eqref{eq:no_adapted_solutions} on a fixed probability space equipped with a two--dimensional Brownian motion
$B$ with respect to some filtration $\{\mathcal{F}_t\}_{t \geq 0}$.
We claim that  if $B_{1,1} > 0$ then $X_{2, t}, t \in [0, 1)$ 
must stay non--positive, while if
$B_{1,1} < 0$ then $X_{2, t}, \ t \in [0, 1)$ must stay non--negative.
Indeed, it is easy to see that if any of these conditions are violated then we have $X_{1, 1} \cdot X_{2, 1} > 0$ and this
contradicts the fact that  SDE \eqref{eq:no_solutions} has no solutions. Now it is easy to see that the random variable $X_{2, 1/2}$ cannot be measurable with respect to $\mathcal{F}_{1/2}$ because
the random variable 
$\operatorname{sgn}(B_{1, 1}) = -\operatorname{sgn}(X_{2,1/2})$ is not measurable with respect to $\mathcal{F}_{1/2}$.

Now let $\Omega_1$ be the set of Brownian trajectories of full measure
such that for every $B\in\Omega_1$ both positive and negative \textit{strong} solutions to SDE~\ref{eq:2_0_1} are defined.
Let $\Omega_2$ be the set of full measure on which the \textit{strong} solution to SDE~\ref{eq:1_2_3} is defined.
Let $\Omega_3$ be the set of full measure
on which the \textit{strong} solution to SDE~\ref{eq:2_2_3} exists
provided $\operatorname{sgn}(X_{1, 1}) = - \operatorname{sgn}(X_{2, 1})$.
Set $\Omega := \Omega_1 \cap \Omega_2 \cap \Omega_3$.
It is clear that $P(\Omega) = 1$ and for every $B \in \Omega$
there exists a function $t \mapsto X_t$ which solves  integral equation \eqref{eq:no_adapted_solutions}.
\end{proof}
\begin{remark}\label{re:measurability_remark}
Since the process $X_{2, t}$ never changes its sign for $t \in (0, 2]$
the same arguments show that for any solution $(B, X)$ to integral equation \eqref{eq:no_adapted_solutions} the following equality holds:
\begin{equation}\label{eq:measurability_identity}
\operatorname{sgn}(B_{1, 1}) = \lim_{n \to \infty}-\operatorname{sgn}(X_{2, 1/n})
\end{equation}
\end{remark}
\begin{theorem}\label{th:no_adapted_solutions_inf}
Let $B$ be a cylindrical Brownian motion. There exists a Borel mapping
$b: [0, 1] \times \mathbb{R}^{\infty} 
\to \mathbb{R}^{\infty}$ 
\begin{equation}\label{eq:no_adapted_solutions_inf}
X_{t} = \int_{[0, t]}b(s, X_s)\, ds + B_t,\ X_0 = 0, \ t \in [0, 1],
\end{equation}
such that 
\begin{enumerate}
\item there exists a set of Brownian trajectories $\Omega \subset C( [0,1], \R^\infty)$
of full measure, such that for every $B \in \Omega$  integral equation \eqref{eq:no_adapted_solutions_inf} has at least one solution defined on the whole interval $[0, 1]$,
\item  equation \eqref{eq:no_adapted_solutions_inf} considered as a SDE
has no solutions $(X, B)$ at all, even defined up to some Markov moment
$\tau$ if $\tau > 0$ on the set of positive probability. 
\end{enumerate}
\end{theorem}
\begin{proof}
One can modify the construction from Theorem \ref{th:no_adapted_solutions} to obtain an equation for which
the existence of adapted solutions does not hold on the interval 
$[0, 1/n]$. 
Considering a system of such equations with independent $2$-dimensional Brownian motions yields the required example. 
Indeed, the existence of a \textit{path-by-path} solution is obvious since it exists for each $2$-dimensional equation. Let us prove that no adapted solutions to \eqref{eq:no_adapted_solutions_inf} exist.
Let us assume that there is a \textit{weak} solution to SDE \eqref{eq:no_adapted_solutions_inf} defined up to the Markov moment $\tau$ and $\P[\tau > 0] \neq 0$.  
Set $\mathcal{F}_{0+} := \bigcap \limits_{t > 0} \mathcal{F}_{t},$
it is well--known that the Brownian motion
remains independent of the $\sigma$--field $\mathcal{F}_{0+}$.
Since
$$
\lim_{k \to \infty} \P[\tau > 2/k] = \P[\tau > 0 ] > 0,
$$ 
we can find $K \in \N$ such that
$$
\P[A_K] \geq \frac{3}{4} \P[A] > 0,
$$
where
$$
A_K := \{\tau > 2/K\}, \ A := \{\tau > 0\}.
$$
Let us consider the $K$th $2$-dimensional equation and since the number $K$ is fixed 
for the sake of brevity we will still denote its solution by $(X, B)$.
Taking into account Remark \ref{re:measurability_remark} we have the following equality:
$$
\mathbbm{1}_{A_K}\operatorname{sgn}(B_{1, 1/K}) = 
\mathbbm{1}_{A_K} \xi, 
$$
where
$$
\xi := \lim_{n \to \infty}-\operatorname{sgn}(X_{2, 1/n}).
$$
One can notice that $A \in \mathcal{F}_{0+}$, $\xi$ is measurable with respect to $\mathcal{F}_{0+}$ and a.s. $|\xi| = 1$.
We have the following chain of equalities:

\begin{equation*}
\begin{split}
\mathbbm{1}_{A}\xi &= 
\mathbb{E}\bigl[\mathbbm{1}_{A_K} \xi | \mathcal{F}_{0+}\bigr] +
\mathbb{E}\bigl[(\mathbbm{1}_{A} -\mathbbm{1}_{A_K}) \xi | \mathcal{F}_{0+}\bigr]\\
 &=  \mathbb{E}\bigl[\mathbbm{1}_{A_K} \operatorname{sgn}(B_{1, 1/K}) | \mathcal{F}_{0+}\bigr]  +
\mathbb{E}\bigl[(\mathbbm{1}_{A} -\mathbbm{1}_{A_K}) \xi | \mathcal{F}_{0+}\bigr] \\
&= 
\mathbb{E}\bigl[\mathbbm{1}_{A} \operatorname{sgn}(B_{1, 1/K}) | \mathcal{F}_{0+}\bigr] 
 +
\mathbb{E}\bigl[(\mathbbm{1}_{A_K} -\mathbbm{1}_{A}) \operatorname{sgn}(B_{1, 1/K}) |  \mathcal{F}_{0+}\bigr] 
 +
\mathbb{E}\bigl[(\mathbbm{1}_{A} -\mathbbm{1}_{A_K}) \xi | \mathcal{F}_{0+}\bigr] \\           
&= 
\mathbbm{1}_{A}\mathbb{E}\bigl[\operatorname{sgn}(B_{1, 1/K}) | \mathcal{F}_{0+}\bigr] +
\mathbb{E}\bigl[(\mathbbm{1}_{A_K} -\mathbbm{1}_{A}) \operatorname{sgn}(B_{1, 1/K}) |  \mathcal{F}_{0+}\bigr] 
 +
\mathbb{E}\bigl[(\mathbbm{1}_{A} -\mathbbm{1}_{A_K}) \xi | \mathcal{F}_{0+}\bigr]\\
&=
\mathbb{E}\bigl[(\mathbbm{1}_{A_K} -\mathbbm{1}_{A}) \operatorname{sgn}(B_{1, 1/K}) |  \mathcal{F}_{0+}\bigr] 
 +
\mathbb{E}\bigl[(\mathbbm{1}_{A} -\mathbbm{1}_{A_K}) \xi | \mathcal{F}_{0+}\bigr],
\end{split}
\end{equation*}
where we have used the fact that
$$
\mathbb{E}\bigl[\operatorname{sgn}(B_{1, 1/K}) | \mathcal{F}_{0+}\bigr]  = 0.
$$
Then:
\begin{equation*}
\begin{split}
\P[A] &=\mathbb{E} \mathbbm{1}_{A}|\xi| \\
&\leq 
\mathbb{E}\Bigl | 
\mathbb{E}\bigl[(\mathbbm{1}_{A_K} -\mathbbm{1}_{A}) \operatorname{sgn}(B_{1, 1/K}) |  \mathcal{F}_{0+}\bigr] \Bigr| +
\mathbb{E}\Bigl | 
\mathbb{E}\bigl[(\mathbbm{1}_{A} -\mathbbm{1}_{A_K}) \xi | \mathcal{F}_{0+}\bigr]
\Bigr | \\
&\leq 2 \P[A \setminus A_K] \leq \frac{1}{2}\P[A].
\end{split}
\end{equation*}
The established inequality contradicts to the assumption $\P[A] = \P[\tau > 0 ] > 0$.
Now it is easy to complete the proof.
\end{proof}
\section{Pathwise uniqueness without path-by-path uniqueness}

\begin{theorem}\label{th:no_path_by_path_uniqueness}
There exists a Borel mapping $b = (b_1, b_2, b_3): [0, 2] \times \mathbb{R}^3 \to \mathbb{R}^3$,
\begin{equation}\label{eq:no_path_by_path_uniqueness}
X_{t} = \int_{[0, t]}b(s, X_s)\, ds + B_t,\ X_0 = 0, \ t \in [0, 2]
\end{equation}
such that
\begin{enumerate}
\item there exists a set of Brownian trajectories $\Omega \subset C( [0,2], \R^3)$
of full measure, such that for every $B \in \Omega$  integral equation \eqref{eq:no_path_by_path_uniqueness} has at least two different solutions (understood in the pure ODE sense) defined on the whole  interval $[0, 2]$.
\item 
There exists a \textit{weak} solution to the corresponding SDE defined on the whole interval $[0, 2]$, moreover, this solution is \textit{strong} and \textit{pathwise} uniqueness holds.
\end{enumerate}
\end{theorem}
\begin{proof}
For $t \in [0, 1]$ let us set
$$
b_1 := 0,
$$
$$
b_2 := \mathbbm{1}_{\{x_2 > 0\}}\mathbbm{1}_{\{x_3 > 0\}}
\Bigl(
\frac{1 - x_2}{1 - t} + \frac{1}{x_2}
\Bigr)
- 
\mathbbm{1}_{\{x_2 < 0\}}\mathbbm{1}_{\{x_3 > 0\}}
\Bigl(
\frac{1 - (-x_2)}{1 - t} + \frac{1}{-x_2}
\Bigr),
$$
$$
b_3 :=
\mathbbm{1}_{\{x_3 > 0\}}
\Bigl(
\frac{1 - x_3}{1 - t} + \frac{1}{x_3}
\Bigr)
- 
\mathbbm{1}_{\{x_3 < 0\}}
\Bigl(
\frac{1 - (-x_3)}{1 - t} + \frac{1}{-x_3}
\Bigr).
$$

For $t \in (1, 2]$ let us set
$$
b_1 := \mathbbm{1}_{\{x_1 \neq 0\}} \mathbbm{1}_{\{x_3 > 0\}}
 \frac{1}{x_1},
$$
\begin{multline*}
b_2 := \mathbbm{1}_{\{x_1 > 0\}} \mathbbm{1}_{\{x_2 > 0\}} \mathbbm{1}_{\{x_3 > 0\}}
\frac{-1}{2(x_2 - 1)} +
\mathbbm{1}_{\{x_1 < 0\}}\mathbbm{1}_{\{x_2 < 0\}} \mathbbm{1}_{\{x_3 > 0\}}
\frac{-1}{2(x_2 + 1)} \\
+\mathbbm{1}_{\{x_1 > 0\}}\mathbbm{1}_{\{x_2 < 0\}} \mathbbm{1}_{\{x_3 > 0\}} \frac{1}{x_2} + 
\mathbbm{1}_{\{x_1 < 0\}}\mathbbm{1}_{\{x_2 > 0\}} \mathbbm{1}_{\{x_3 > 0\}} \frac{1}{x_2} ,
\end{multline*}
$$
 b_3 :=  \mathbbm{1}_{\{x_3 \neq 0\}}  \frac{1}{x_3}.
$$

First, let us present a \textit{strong} solution to  equation \eqref{eq:no_path_by_path_uniqueness}.
On the interval $[0,1]$ let $X_3$ be the $(-1)\times 3$--Bessel bridge
from $0$ to $-1$ and on the interval $[1, 2]$ let $X_3$ be the 
$(-1)\times 3$--Bessel process with  the initial condition $X_{3,1} = -1$. Applying Theorem \ref{th:cherny_uniqueness_theorem} and Proposition \ref{pr:helper_equation_properties} one can see that 
$X_3$ is indeed adapted to the filtration generated by the Brownian motion $B$ and negative for all $t \in (0, 2]$.
Let us set $X_1 := B_1$ and $X_2 = B_2$ for $t \in [0, 2]$.
It is clear that since $X_{3, t}$ is nonpositive for $t \in [0, 2]$ then
the drifts $b_1, b_2$ are identically equal to zero, therefore
the constructed process $(X_1, X_2, X_3)$ is indeed a \textit{strong}
solution to  SDE \ref{eq:no_path_by_path_uniqueness}.

Second, one can notice, that if we take the positive ``version'' of $X_3$, e.g.
on the interval $[0, 1]$ we can define $X_3$ as the $3$--Bessel bridge
from $0$ to $1$ and on the interval $[1, 2]$ let $X_3$ be the 
$3$--Bessel process with  the initial condition $X_{3,1} = 1$, then
the drifts $b_1, b_2$ coincide with the mappings constructed in Theorem \ref{th:no_adapted_solutions}, thus in this case the equation 
for $(X_1, X_2)$ has only non--adapted solutions, but the set of solutions is nonempty. We have shown that with probability one 
there are at least two different solutions to  integral equation
\eqref{eq:no_path_by_path_uniqueness} (corresponding to the positive and negative ``versions'' of $X_3$) and \textit{path-by-path} uniqueness does not hold.

Finally, let us establish \textit{pathwise} uniqueness for equation
\eqref{eq:no_path_by_path_uniqueness}.
Let $$X = (X_1, X_2, X_3)$$ be a \textit{weak} solution to stochastic differential equation \ref{eq:no_path_by_path_uniqueness} on the interval $[0, 2]$ with some Brownian motion $B$.
We would like to show that a.s. $X$ coincides with the strong solution
presented above. Applying Theorem \ref{th:cherny_uniqueness_theorem} and Proposition \ref{pr:helper_equation_properties} one may notice that
a.s. $X_3$ does not change its sign
for $t \in (0, 2]$. 
The arguments from the proof of Theorem \ref{th:no_adapted_solutions} show that on the set 
$$
U := \{X_{3, 1/2} > 0\} \in \mathcal{F}_{1/2}
$$ 
a.s.~we have the equality
$$\operatorname{sgn}(B_{1, 1}) = -\operatorname{sgn}(X_{2, 1/2}).$$
Then:  
\begin{equation}\label{eq:almost_contradiction}
\begin{split}
\mathbb{E}\bigl[\operatorname{sgn}(B_{1, 1})|\mathcal{F}_{1/2}\bigr] &= 
\mathbb{E}\bigl[\mathbbm{1}_{U}\operatorname{sgn}(B_{1, 1})|\mathcal{F}_{1/2}\bigr] + 
\mathbb{E}\bigl[\mathbbm{1}_{\Omega\setminus U}\operatorname{sgn}(B_{1, 1})|\mathcal{F}_{1/2}\bigr] \\
 &= -\mathbbm{1}_{U}\operatorname{sgn}(X_{2, 1/2}) + 
 \mathbbm{1}_{\Omega\setminus U} \mathbb{E}\bigl[\operatorname{sgn}(B_{1, 1})|\mathcal{F}_{1/2}\bigr].
\end{split}
\end{equation}
But at the same time
$$
\mathbb{E}\bigl[\operatorname{sgn}(B_{1, 1})|\mathcal{F}_{1/2}\bigr] = 
[P_{1/2}\operatorname{sgn} ](B_{1, 1/2}),
$$
where $(P_{t})_{t \geq 0}$ is the standard heat semigroup 
defined by the formula
$$
P_{t}f (x) = \int_{\mathbb{R}}f(y
)\frac{1}{\sqrt{2\pi}}e^{-\frac{(x- y)^2}{2t}}\, dy.
$$
It easy to see that for almost all $x \in \mathbb{R}$ we have the strict inequalities
$$0 < | P_{1/2}\operatorname{sgn} (x) | < 1,$$
consequently, the equality
$$
[P_{1/2}\operatorname{sgn} ](B_{1, 1/2}) = -\mathbbm{1}_{U}\operatorname{sgn}(X_{2, 1/2}) + 
 \mathbbm{1}_{\Omega\setminus U} \mathbb{E}\bigl[\operatorname{sgn}(B_{1, 1})|\mathcal{F}_{1/2}\bigr]
$$
can hold a.s.~only if $\mathbb{P}[U] = 0$.
This means that a.s. $X_3$ is negative for $t \in (0, 2]$
and now it is trivial to complete the proof.
\end{proof}

\begin{theorem}\label{th:no_path_by_path_uniqueness_inf}
Let $B$ be a cylindrical Brownian motion. There exists a Borel mapping $b: [0, 1] \times \mathbb{R}^{\infty} \to \mathbb{R}^{\infty}$,
\begin{equation}\label{eq:no_path_by_path_uniqueness_inf}
X_{t} = \int_{[0, t]}b(s, X_s)\,ds + B_t,\ X_0 = 0, \ t \in [0, 1]
\end{equation}
such that
\begin{enumerate}
\item there exists a set of Brownian trajectories $\Omega \subset C( [0,2], \R^\infty)$
of full measure such that for every $B \in \Omega$  integral equation \eqref{eq:no_path_by_path_uniqueness_inf} has at least two different solutions (understood in the pure ODE sense) defined on the whole  interval $[0, 1]$.
\item \textit{pathwise} uniqueness holds in the sense that 
for any Markov moment $\tau$ there exists
a unique weak solution to  SDE \eqref{eq:no_path_by_path_uniqueness_inf} defined up to the moment~$\tau$, this solution is strong and pathwise uniqueness holds.
\end{enumerate}
\end{theorem}
\begin{proof}
Let $\widetilde{b}$ be the drift constructed in the proof of Theorem \ref{th:no_adapted_solutions_inf}.
Let us define the new drift $b = (b_1, b_2)$ as follows:
$$
b_1: [0, 1] \times \mathbb{R}^{\infty} \to \mathbb{R}^{\infty},
$$
$$
b_{1} := \mathbbm{1}_{\{x_2 > 0\}}\widetilde{b}(x_1), \ x_1 \in \R^{\infty}.
$$
$$
b_2 :  [0, 1] \times \mathbb{R} \to \mathbb{R},
$$
$$
b_2 := \mathbbm{1}_{\{x_2 > 0\}}
\Bigl(
\frac{1 - x_2}{1 - t} + \frac{1}{x_2}
\Bigr)
- 
\mathbbm{1}_{\{x_2 < 0\}}
\Bigl(
\frac{1 - (-x_2)}{1 - t} + \frac{1}{-x_2}
\Bigr), \ x_2 \in \R.
$$
Then the same arguments as in the proof of Theorem \ref{th:no_adapted_solutions_inf}, Theorem \ref{th:no_path_by_path_uniqueness} show that the drift $b$
has the  desired properties. 
\end{proof}

\section{Acknowledgement}
The authors would like to thank
V.I. Bogachev, M. R\"{o}ckner and L. Mytnik for
some fruitful discussions and 
the anonymous referee for the thorough reading and corrections which have greatly helped improve the manuscript.

This research was supported by the CRC 1283 at Bielefeld University.

\end{document}